\documentclass[11pt,a4paper]{amsart}
\usepackage{amsthm}
\usepackage{amsmath}
\usepackage{amsfonts}
\usepackage{amssymb}
\usepackage{calligra}
\usepackage{color}
\usepackage[T1]{fontenc}

\newcommand{\egon}{{\rm null}}
\newcommand{\egono}{{\rm null_0}}

\newcommand{\cri}{{\sf cr}}
\newcommand{\crl}{{\sf crl}}
\newcommand{\crim}{{\sf crl}\left(M,g\right)}
\newcommand{\crimp}{{\sf crl}_p\left(M,g\right)}
\newcommand{\cris}{{\sf crl}\left(S^n,g\right)}
\newcommand{\crisp}{{\sf crl}_p\left(S^n,g\right)}
\newcommand{\beq}{\begin{equation}}
\newcommand{\beqn}{\begin{equation*}}

\newcommand{\diam}{{\rm diam}}

\newcommand{\eeq}{\end{equation}}
\newcommand{\eeqn}{\end{equation*}}
\newcommand{\field}{\mathbb{Z}_s}

\newcommand{\grad}{{\rm grad}}
\newcommand{\homo}{{\zeta}}

\newcommand{\indo}{{\rm ind_0}}
\newcommand{\ind}{{\rm ind}}
\newcommand{\inj}{{\rm inj}}

\newcommand{\n}{\mathbb{N}}
\newcommand{\rad}{{\rm rad}}
\newcommand{\R}{\mathbb{R}}

\newcommand{\lcg}{{\mathcal L}}
\newcommand{\lglp}{{\mathcal L}(p)}

\newcommand{\z}{\mathbb{Z}}
\newtheorem{theorem}{Theorem}
\newtheorem{lemma}{Lemma}
\newtheorem{corollary}{Corollary}
\newtheorem{definition}{Definition}
\newtheorem{example}{Example}
\newtheorem{remark}{Remark}
\newtheorem{proposition}{Proposition}
\title[Critical values and positive curvature]{Critical 
values of homology classes of loops and 
positive curvature}
\author{Hans-Bert Rademacher}
\address{Mathematisches Institut, Universit{\"a}t Leipzig,
04081 Leipzig, Germany}
\email{rademacher@math.uni-leipzig.de}
\date{2017-10-27}
\begin{document}
\begin{abstract}
We study compact and
simply-connected Riemannian manifolds $(M,g)$
with
positive sectional curvature $K\ge 1.$
For a non-trivial homology class
of lowest 
positive dimension 
in the space of loops based at a
point $p \in M$ or in the free loop space
one can define a critical length
$\crimp$ resp. $\crim.$ Then
$\crimp$ equals the length of a geodesic loop
with base point $p$
and $\crim$ equals the length of a closed
geodesic.
This is the idea of the 
proof of the existence of a 
closed geodesic of positive length
presented by Birkhoff
in case of a sphere and by Lusternik \&
Fet in the general case. 
It is the main result of the paper that
the numbers $\crimp$ resp. $\crim$
attain its maximal value $2\pi$ only
for the round metric on the $n$-sphere.

Under the additional assumption $K \le 4$ this result for $\crim$ follows from
results by Sugimoto in even dimensions
and Ballmann, Thorbergsson \& Ziller in
odd dimensions. 
\end{abstract}
\maketitle
\section{Introduction}
\label{sec:introduction}

For  a compact Riemannian manifold $(M,g)$
let 
$\mathcal{M}=\{ \gamma: [0,1]\longrightarrow M\,;\,
\gamma \mbox{ absolutely continuous },
\int_0^1 \|\gamma'\|^2(t)\,dt<\infty\}$
be the manifold of paths ($H^1$-paths) on $M.$ We consider in
the sequel the following subspaces:
the {\em free loop space}
$\Lambda = \Lambda M= \left\{\gamma \in \mathcal{M}; \gamma(0)=\gamma(1)\right\}$
and for $p,q \in M$
the space of paths between $p$ and $q:$
$\Omega_{pq}=\Omega_{pq}M=
\left\{\gamma \in \mathcal{M}; \gamma(0)=p, \gamma(1)=q\right\}.$
In particular for $p=q$ one obtains the 
{\em (based) loop space}  $\Omega_p M=\Omega_{pp}M,$
 cf.~\cite[ch.2.3]{Klingenberg1995}.
Now we use $X$ to denote one of the spaces $\Lambda M, \Omega_{pq}M, \Omega_pM.$
We consider the energy functional $E: X \longrightarrow \R:$
\begin{equation}
 \label{eq:definition-energy}
 E\left(\gamma\right)=\frac{1}{2}\int_0^1 
\|\gamma'(t)\|^2(t)\,dt\,.
 \end{equation}
 This is a differentiable function, the critical points are
 closed geodescis (if $X=\Lambda M$), geodesics joining $p$ and $q$ (if
 $X=\Omega_{pq}M$) and geodesic loops based at $p$ (if
  $X=\Omega_{p}M$). 
 
  For a critical point $c \in X$ we have $E(c)=L^2(c)/2,$ here
  $L(c)=\int_0^1\|c'(t)\| \, dt$ is the {\em length} of $c.$
We use the following notation
for sublevel sets:
  $X^{\le a}:=
  \left\{\gamma \in X\,;\,E(\gamma)\le a\right\}$ resp. $
  X^{<a}:=
  \left\{\gamma \in X\,;\,E(\gamma)< a\right\}$
  or $X^{a}:=
  \left\{\gamma \in X\,;\right.$ 
	$\left.\,E(\gamma)= a\right\}.$
  
  The {\em critical value} (or {\em critical level})
$\cri(h)$ of a relative homology class
$h \in H_k\left(X,X^{\le b},\field\right)$ is defined as follows:
\begin{equation}
 \label{eq:critical-value}
\cri(h)= \inf \left\{a \ge b\,;\,h \in {\rm Image}
\left(H_{k}(X^{\le a}, X^{\le b};\field)\longrightarrow 
H_{k}(X, X^{\le b};\field)
\right)
\right\}\,,
\end{equation}
cf. \cite[ch.2.1]{Klingenberg1978} or \cite[Sec.3]{Hingston2013}.
Then $h \in H_k\left(X, X^{\le b};\field\right)$ is nontrivial
if and only if $\cri(h)>b.$
For a non-trivial class $h \in H_k\left(X;\field\right)$
there is a critical point $c \in X$ with $E(c)=\cri(h)$
resp. $L(c)=\sqrt{2 E(c)}.$

For a compact and simply-connected manifold $M=M^n$ of dimension $n$ 
let $k=k(M) \in \{1,2,\ldots,n-1\}$ be the unique number, such that 
$M$ is $k$-connected, but not $(k+1)$-connected. Hence
we have the following statement about the homotopy groups:
$\pi_j(M)=0$ for $1\le j\le k$ and 
$\pi_{k+1}(M)\not=0\,.$ 
Then for $X=\Omega_{pq}M$ resp.
$X=\Lambda M$ there is a smallest prime number 
$s \in \mathbb{P}$ such that 
$H_{k}\left(X, X^0;\field\right)\not=0.$ 
This follows since
$\pi_j\left(\Omega_{p}M\right)
\cong \pi_{j+1}\left(M\right) $
resp. $\pi_j\left(\Lambda M\right)\cong 
\pi_{j+1}\left(M\right) \oplus \pi_j\left(M\right),$
cf.~\cite[Cor. 2.1.5]{Klingenberg1978}.
We define the 
{\em critical length} $\crim,$ 
resp. the {\em critical length $\crimp $
at $p$} as follows:
\begin{equation}
\label{eq:critical-value-manifold}
\crim:=
\max\left\{\sqrt{2 \cri(h)}\,;\, h \in 
H_{k(M)}\left(\Lambda M, \Lambda^0 M;\field\right)
\right\}
\end{equation}
and 
\begin{equation}
\label{eq:critical-value-manifold-p}
\crimp :=
\max\left\{\sqrt{2\cri(h)}\,;\, h \in 
H_{k(M)}
\left(\Omega_p M;\field\right)\right\}\,.
\end{equation}
If the manifold $M$ is  simply-connected and
homeomorphic to a sphere $S^n$ then there is a uniquely determined generator
$ h \in H_{n-1}\left(\Lambda M,
\Lambda^0 M;\z_2\right)\cong \z_2,$
resp. $h_p \in H_{n-1}\left(\Omega_pM;\z_2\right)
\cong \z_2$ such that $\crim=\cri (h),$ resp. 
$\crimp=\cri (h_p).$
For a metric of positive sectional curvature
we obtain the following result for the critical
length resp. the critical length at a point. In
particular we show that the maximal possible
value is only attained for the round metric
on a sphere of constant sectional curvature $1:$
\begin{theorem}
\label{thm:one}
Let $\left(M^n,g\right)$ be a 
compact and simply--connected Riemannian
manifold of
positive sectional curvature $K \ge 1.$
Then the critical lengths $\crim$ and $\crimp$ for
$p \in M$
satisfy:

\smallskip

(a) 
For all $p \in M:$ $\crimp\le 2 \pi$
and $\crim \le 2\pi.$

\smallskip

(b) If for some $p \in M:$
$\crimp > \pi$ 
or if $\crim >\pi$ 
then the manifold is homeomorphic to
the $n$-dimensional sphere $S^n.$

\smallskip

(b) If $\crimp= 2 \pi$
for some $p \in M,$ 
or if $\crim =2\pi$ then $(M,g)$ is isometric to
the round sphere of constant sectional curvature $1.$
\end{theorem}
\begin{remark}\rm
\label{rem:diameter}
(a) Note that the conclusions of the Theorem remain true if we
replace $\crim$ by $ 2\, \diam (M,g),$ where $\diam (M,g)$ is the
{\em diameter} of the Riemannian manifold.
Then Part (a) is Bonnet's theorem, cf.~\cite[Thm.1.31]{Cheeger1975},
Part (b) is the generalized sphere theorem~\cite{Grove1977}
by Grove \& Shiohama 
and Part (c) is Toponogov's maximal diameter 
theorem~\cite[Thm. 6.5]{Cheeger1975},
resp.~\cite{Toponogov1958}.
But one cannot reduce the proof of the Theorem to the
statements for the diameter as examples
constructed by
Balacheff, Croke \& Katz~\cite{Balacheff2009}
show:
There is a one-parameter family $g_t,t\ge 0$ of smooth
Zoll metrics on the $2$-sphere such that the following
holds: The metric $g_0$ is the round metric of constant
curvature $1$ and for
$t>0$ 
the length $
\lcg (S^2,g_t)$ of a shortest
closed geodesic satisfies $\lcg (S^2,g_t) > 2 \,\diam (S^2,g_t).$

(b) Since $\crim>0$ one obtains the existence
of a non-trivial closed geodesic
of length $\crim,$ resp.
a geodesic loop of length $\crimp.$ This idea
goes back to Birkhoff~\cite{Birkhoff1917} in the case of a sphere and
to Lyusternik-Fet~\cite{Lyusternik1951} in the general case.
Therefore we call $\crim$ also the {\em Birkhoff-length} of the
Riemannian manifold.

(c) Since for the number $k=k(M)$ we have
$\pi_k\left(\Lambda M, \Lambda^0 M\right)  
\cong \pi_k\left(\Omega_pM\right)\cong
\pi_{k+1}(M)$ and since for any $p \in M$
the based loop space $\Omega_pM$ can be seen as
a subspace of the free loop space $\Lambda M$
we obtain for all $p \in M:$
\begin{equation}
\label{eq:crlp}
\crim \le \crimp.
\end{equation}
In Example~\ref{exa:convex} we construct 
a one-parameter family $(S^2,g_r), 0<r\le 1$
of convex surfaces  of revolution with
$\crl_p (S^2,g_r)\ge \pi$ for all $p \in S^2,$
but $\lim_{r\to 0}\crl (S^2,g_r)=0.$
\end{remark}
For a compact Riemannian manifold $(M,g)$ 
we denote by $\lcg=\lcg(M,g)$ the length of a shortest (non-trivial) closed geodesic.
For a point $p \in M$ 
we denote by
$\lglp$ the length of a shortest (non-trivial) geodesic
loop with initial point $p.$
Then we have the obvious estimates:
$\lcg\le \crim$ and $\lglp \le \crimp.$ 
In particular we obtain as 
\begin{corollary}
Let $\left(M^n,g\right)$ be a simply--connected Riemannian
manifold of
positive sectional curvature $K \ge 1$
and $p \in M.$

(a) $\lcg \le 2\pi,$ and $\lglp \le 2\pi$ for all $p \in M.$

(b) If $\lcg =2\pi$ or if $\lglp = 2\pi $ for some
$p \in M$ then the manifold is isometric to the 
sphere of constant curvature $1.$
\end{corollary}
Toponogov proves in~\cite{Toponogov1959}
that a {\em simple} closed geodesic (i.e. 
a closed geodesic with no self-intersection)
on a convex surface with $K\ge 1$ has length $\le 2\pi$
with equality if and only if the sectional curvature is constant.
He also shows the analogous statement for a geodesic loop.
Statement (b) for $\lcg$ is claimed in several 
preprints (unpublished) by Itokawa \& 
Kobayashi~\cite{Itokawa1991},
\cite{Itokawa2005}, \cite{Itokawa2008}. 
In case of quaterly pinched metrics there are 
stronger results about a gap in the length
spectrum obtained using the 
Toponogov triangle comparison results:
\begin{theorem}
 \label{thm:Ballmann-Habil}
Let $M$ be a manifold homeomorphic to $S^n$ carrying  a Riemannian
 metric $g$ with sectional curvature $1 \le K \le \Delta \le 4.$
 
 (a) {\rm (Ballmann~\cite[Teil III]{Ballmann-Habil})}
 The critical length $\crim$ 
and the length $\lcg$
of a shortest closed geodesic coincide and satisfy:
$ 2 \pi/\sqrt{\Delta} \le \crim=\lcg \le 2\pi$.

(b) {\rm (Ballmann,Thorbersson \& 
Ziller~\cite[Thm.1.7]{BTZ1983c})}
If there is a closed geodesic $c$ with length
$2\pi \le L(c)\le 4\pi/\sqrt{\Delta}$
then the metric has constant sectional
curvature.
\end{theorem}
\begin{remark}
\rm
(a) The case $n=2$ in Part (a) is also contained
in~\cite[Thm. 4.2]{BTZ1983c}.
For even dimension Part (b) was shown
by Sugimoto~\cite[Thm. B,C]{Sugimoto1970}. 
compare also the Remark 
following~\cite[Thm.1.7]{BTZ1983c}.

(b) In~\cite[Teil III]{Ballmann-Habil}
the following stronger statement is shown:
For any
shortest closed geodesic $c$ there is a homotopically non-trivial
map \\$G: \left(D^{n-1},\partial D^{n-1}\right)
\longrightarrow \left(\{c\}
\cup \Lambda^{< \lcg^2/2}M,
\Lambda^{< \lcg^2/2}M\right).$
Hence the generator $h \in 
H_{n-1}\left(\Lambda M, \Lambda^0M;\z_2\right)$
{\em remains hanging} at any shortest closed geodesic.

(c)
In~\cite{BTZ1983a} the existence of 
$g(n)$ geometrically distinct closed geodesics 
on $S^n$ is shown
under the assumptions of 
Theorem~\ref{thm:Ballmann-Habil}.
The lengths of these closed geodesics 
lie in the interval 
$\left[2\pi/\sqrt{\Delta}, 2\pi\right]$
and $g(n)$ is the cup-length of the 
Grassmannian $G_2(\R^{n+1})$ of unoriented
$2$-planes in $\R^{n+1}.$
Instead of an assumption about the sectional
curvature, which is an assumption depending
continously on the metric and its first
and second derivatives one can also
investigate assumptions depending continuously
only on the metric. For example Ballmann,
Thorbergsson \& Ziller consider 
in~\cite[Sec.3]{BTZ1983a}
the so-called {\em Morse-condition}.
Hingston shows in~\cite[Sec.5]{Hingston1984}
a similar result about the existence of
$g(n)$ closed geodesics.
These results motivate 
Proposition~\ref{pro:metric}, where we 
consider Riemannian metrics $g$ on $S^n$ satisfying
$g\le g_1.$  Here 
$g_1$ is the round metric of constant curvature
$1.$
\end{remark}
\section{Rauch comparison result}
\label{sec:rauch}
Let $c:\R\longrightarrow M$ be a geodesic
parametrized by arc length, then 
$c(t_1)$ is called {\em focal point} of $p=c(0)$ 
(along $c\left|[0,t_1]\right.$) if 
there is a non-trivial normal Jacobi field $Y=Y(t)$ along $c$ with 
$\langle Y(t),c'(t)\rangle=0$ 
for all $t$ 
and 
vanishing covariant derivative
$\nabla Y(0)/dt=0$ (along the curve)
and $Y(t_1)=0,$
cf.~\cite[Def. 1.12.20]{Klingenberg1995}. The {\em focal radius}
$\rho=\rho(M,g)$ of a compact Riemannian manifold $(M,g)$ equals
the minimal 
positive number $t_1$ with the above property.
The following statement is also called {\em second Rauch comparison
theorem (Rauch II)},
cf.~\cite[Thm.1.34]{Cheeger1975},
which is due to Berger. But 
note that we also include a rigidity 
statement, see for example~\cite[Thm.3.3]{Eschenburg1987}.
\begin{proposition}
\label{pro:rauch}
Let $c:[0,a]\longrightarrow M$ be a 
geodesic parametrized by arc length on a Riemannian
manifold with positive sectional curvature $K \ge 1$
and let $c(t_1)$ be the first focal point.
Let $Y=Y(t)$ be a Jacobi field along $c=c(t)$
with covariant derivative along $c$:
$Y'(0)=\nabla Y(0)/dt=0.$ 
Then $\|Y(t)\|\le \|Y(0)\| \cos(2\pi t)$
for all $0 \le t \le t_1.$
If $\|Y(t_0)\|=\|Y(0)\|\cos(2\pi t_0)$ for some
$t_0 \in (0,t_1],$ then 
$\|Y(t)\|=\|Y(0)\|\cos(2\pi t)$ 
for all $t \in [0,t_0]$ and $Y/\|Y\|$ is parallel
along
$c\left|[0,t_0]\right..$
\end{proposition}
Using this Rauch comparison result one obtains
the following consequence. This result can be
found for example 
in~\cite[Cor.1.36]{Cheeger1975}
or in~\cite[Cor. 2.7.10]{Klingenberg1995}.
But note that we also include a rigidity statement.
We denote by $\exp: TM \longrightarrow M$
the {\em exponential map,}
i.e. 
for $v \in T_pM$ the curve
$t\in \R \longmapsto \exp(t v)$
equals the geodesic $\gamma_v=\gamma_v(t)$
with initial point $p=\gamma_v(0)$ and
direction $v=\gamma'_v(0).$
And $\exp_p:T_pM \longrightarrow M$ denotes the
restriction to the tangent space $T_pM$ at $p.$
\begin{corollary}
\label{cor:rauch}
Let $c: [0,a]\longrightarrow M$ be 
a geodesic on a compact 
Riemannian manifold of positive sectional
curvature $K \ge 1$ 
and focal radius $\rho>0.$
Let $c_1: [0,a]\longrightarrow S^n_1$ be a great
circle of length $L(c)$ on the standard
sphere $S^n_1$ of constant sectional curvature $1.$
Choose 
parallel unit vector fields $E$ along $c$ resp.
$E_1$ along $c_1$ which are orthogonal to the geodesic
$c$ resp. $c_1.$
For a smooth 
non-negative function $f:[0,a]\longrightarrow \mathbb{R}$ 
with $0 < f(t) \le \rho$ for all $t \in (0,a)$
we define
the smooth curve
$b(t)=\exp_{c(t)}\left(f(t) E(t)\right) \in M, 
t\in [0,a]$
on $M$ resp. the curve 
$b_1(t)=\exp_{c_1(t)}
\left(f(t) E_1(t)\right)\in S_1^n, t\in [0,a]$
on the round sphere $S^n_1$ of constant sectional
curvature $1.$

Then the following statements hold:

\smallskip

(a) $L(b)\le L(b_1).$ 

\smallskip

(b) 
Assume $L(b)=L(b_1).$
Let $f^*:[0,a]\longrightarrow \R$ be a non-negative
function with $0<f^*(t)\le f(t)$ for all $t \in (0,a)$
and define the curve
$b^*(t)=\exp_{c(t)}
\left(f^*(t) E(t)\right) \in M, t\in [0,a]$
on $M$ resp.
the following curve 
$b^*_1(t)=$ $\exp_{c_1(t)}
\left(f^*(t) E_1(t)\right)\in S^n_1, t\in [0,a]$
on the round sphere $S^n_1.$

Then $L(b^*)=L(b^*_1).$
\end{corollary}
\begin{proof}
The proof of (a) is given in~\cite[Cor. 1.36]{Cheeger1975}
or \cite[Cor. 2.7.10]{Klingenberg1995}.

Assume $L(b)=L(b_1).$ Let for $t \in [0,a]:$
$s\in[0,f(t)]\longmapsto \gamma_t(s)=\exp_{c(t)}(s E(t))$
denote the geodesic with $\gamma_t(0)=c(t)$ and
$\gamma_t'(0)=E(t)$ and let $s \in [0,f(t)]
\longmapsto W_t(s)$ be the parallel
unit vector field along $\gamma_t$ with $W_t(0)=c'(t)$
for all $t \in [0,a].$
We denote by $Y_t(s)=\partial \gamma_t(s)/\partial s$
the Jacobi field  
defined by the geodesic
variation $t \longmapsto \gamma_t,$ 
which
satisfies $\|Y_t(f(t))\|=
\cos(2\pi f(t)).$ It follows from 
the Proof of part (a) that
$\|Y_t(s)\|=\cos (2\pi s)$ for all $s \in [0,f(t)].$
Then we conclude from
Proposition~\ref{pro:rauch}
that 
$W_t(s)=Y_t(s)/\cos(2\pi s)$ is the parallel unit vector field
along $c_t:[0,f(t)]\longrightarrow M$ with
$W_t(0)=c'(t).$
Then $U:=
\left\{\exp_{c_1(t)}\left(s E_1(t)\right)\in S^n_1\,;\,
0 \le s \le f(t)\right\}$ is a 
totally geodesic surface in the round sphere.
Hence the map 
$$
\exp_{c_1(t)}\left(s E_1(t)\right) \in U \subset
S_1^n
\longmapsto \exp_{c(t)}\left(s E(t)\right)\in M$$
defines a
totally geodesic isometric immersion of 
the totally geodesic surface $(U,g_1)$ 
in $S^n_1$ into the Riemannian manifold
$(M,g).$
Therefore the claim follows.
\end{proof}
\section{Morse Theory on path spaces}
\label{sec:morse}
The energy functional $E: \Omega_{pq}M \longrightarrow \R$
is differentiable and its critical points 
are geodesics $c \in \Omega_{pq}M.$
The maximal dimension of a subspace of the 
tangent space $T_c\Omega_{pq}M$ on which the 
hessian $d^2 E(c)$ is negative definite, is called 
{\em index} $\indo (c).$
If $p=q$ and if $c$ is a closed geodesic, then its
index $\ind (c)$ as critical point of $E: \Lambda M
\longrightarrow \R$ is analogously defined.
In general the invariants $\indo (c)$ and $\ind (c)$
differ, the difference
$0 \le {\rm conc} (c):=\ind(c)-\indo (c)\le n-1$ is called 
{\em concavity,}
cf.~\cite[Thm.2.5.14]{Klingenberg1995}.
The nullity 
$\egono (c)$ of a geodesic $c \in \Omega_{pq}M$
is defined as the dimension of the kernel of
the hessian. It coincides with the dimension of Jacobi 
fields $Y=Y(t)$ along $c=c(t),t\in [0,1]$ which
vanish at the end points, i.e. $Y(0)=Y(1)=0.$
It follows that $0 \le \egono(c) \le n-1.$
If $c$ is a closed geodesic $c \in \Omega_p M \subset
\Lambda M$ then the nullity 
$\egon(c)$ as critical point of
the $S^1$-invariant energy functional 
$E: \Lambda M \longrightarrow \R$ equals the dimension
of periodic and orthogonal Jacobi fields.
Hence $0 \le \egon(c) \le 2n-2.$

The energy functional $E:\Omega_{pq}M\longrightarrow
\R$ is a {\em Morse function}, if all geodesics
$c \in \Omega_{pq}M$ are {\em non-degenerate,}
i.e. $\egono (c)=0.$ This is the case if and only if
$q \in M$ is a regular point of the exponential
mapping $\exp_p: T_pM \longrightarrow M,$ 
cf.~\cite[\S 18]{Milnor1963}.
The Morse lemma implies that around any geodesic $c$ there are coordinates
$(x,y) \in V_-\oplus V_+$ with $E(x,y)=-\|x\|^2+\|y\|^2.$
Here $V_-$ is a finite-dimensional subspace 
of dimension $\indo (c)$
on which
the Hessian $d^2E(c)$ is negative definite and
$V_+$ is a closed complement in $T_c\Omega_{pq}M.$
Let the energy functional 
$E:\Omega_{pq}M\longrightarrow \R$
be a Morse function,
$ l \in \n$   and let $a$ be a critical value  
such that for some $\epsilon>0$ the value $a$ is
the only critical value in
the interval $(a-2\epsilon,a+2\epsilon).$
Then the following holds:
There are 
geodesics $c_1,\ldots,c_r \in \Omega_{pq}M$
with $E(c_1)=\ldots=E(c_r)=a$ and
$\indo(c_1)=\ldots=\indo(c_r)=l$
that for all $\delta \in (0,\epsilon]:$
\begin{equation}
\label{eq:h-morse}
H_{l}\left(\Omega_{pq}^{\le a+\delta}M,
\Omega_{pq}^{\le a-\delta}M;\z_2\right)
\cong \bigoplus_{k=1}^r \z_2\left[F_{c_k}\right]\,.
\end{equation}
The relative homology class $[F_{c_k}]\in
H_{l}\left(\Omega^{\le a+\delta}_{pq}M,
\Omega^{\le a-\delta}_{pq}M;\z_2\right)$
can be represented 
by the set
$\{(x,0); x \in V_-\}$ in a neigborhood
of $c_k$ in which Morse coordinates 
$(x,y)\in V_-\oplus V_+$ exist.
The Morse Lemma implies that 
there is a sufficiently small
neighborhood $U(c_k)\subset\Omega_{pq}M$ 
of $c_k$ such that 
the homology class $[F_{c_k}]$ is the unique
generator of the
critical group
$H_{l}\left(\Omega_{pq}^{\le a-\delta}M
\cup U(c_k)\right.$, $\left.
\Omega_{pq}^{\le a-\delta}M;\z_2
\right)\cong \z_2.$

In the sequel we use an approximation argument.
If $E: \Omega_{p}M \longrightarrow \R$ is not
a Morse function, we choose a sequence $p_j$
of regular points for the exponential map
$\exp_p$ converging to $p.$ This is possible since
the regular points of the exponential map are dense,
cf.~\cite[Cor. 18.2]{Milnor1963}.
\section{Critical values of homology classes of loops}
\label{sec:critical-values}
Let $\inj= \inj(M,g)>0$ be the {\em injectivity radius} of the Riemannian
manifold $(M,g).$ Then between two points $p,q$ of distance 
$d(p,q)< \inj$ there is a unique {\em minimal} geodesic $c:[0,1]\longrightarrow M$
with $p=c(0),q=c(1)$ and $L(c)=d(p,q).$
Here $d(p,q)$ is the {\em distance} of the
points $p$ and $q.$
\begin{lemma}
\label{lem:homo}
Let $p,q,r \in M$ with $d(q,r) \le \min \left\{\inj/2,1/2\right\} $ and let
$c_{rq}:[0,d]\longrightarrow M$ be 
the unique minimal geodesic joining $r=c_{rq}(0)=r,
q=c_{rq}(1)$ with $L(c_{rq})=d=d(q,r).$
 Let $\homo_{rq}: \Omega_{pq}M \longrightarrow \Omega_{pr} M$
 be defined by
 \begin{equation*}
 \homo_{rq}(\sigma)(t):=
 \left\{ 
 \begin{array}{ccc}
  \sigma\left(t/(1-d(q,r)\right)&;& t\in [0,1-d(q,r)]\,;\\
  &&\\
  c_{rq}\left(t-\left(1-d(q,r)\right)\right)&;& t\in [1-d(q,r),1]
   \end{array}
\right.\,.
\end{equation*}
Let $h \in H_j\left(\Omega_{pq}M;\z_s\right)$
and let
$\left(\homo_{rq}\right)_*:$
$H_j\left(\Omega_{pq}M;\field\right)
\longrightarrow 
H_j\left(\Omega_{pr}M;\field\right)$
be the induced homomorphism. Then
we obtain:
\begin{equation}
\label{eq:crisigma}
(1-d(q,r))\cri(h)-\frac{1}{2} d(q,r)
 \le \cri\left(\left(\homo_{rq}\right)_{*}(h)\right)
 \le \frac{\cri(h)}{1-d(q,r)}+\frac{1}{2}d(q,r).
\end{equation}
In particular the function
$r \in M \longmapsto 
\cri\left(\left(\homo_{rq}\right)_{*}(h)\right)\in \R^+$
is continuous at $q.$
\end{lemma}
\begin{proof}
For $\sigma:[0,1]\longrightarrow M,$ with $p=\sigma(0), q=\sigma(1)$
 and $r \in M$ with $d(q,r)\le \inj$ 
 we obtain 
 \begin{equation}
 \label{eq:Esigma}
  E\left(\homo_{rq}\left(\sigma\right)\right)
 = \frac{E\left(\sigma\right)}{1-d(q,r)}+
 \frac{1}{2}d(q,r)\,.
 \end{equation}
 We conclude:
 \begin{equation*}
  \cri\left(\left(\homo_{qr}\right)_{*}(h)\right)
  \le \frac{\cri(h)}{1-d(q,r)}+ \frac{d(q,r)}{2}\,.
 \end{equation*}
 Applying this inequality to 
 $\left(\homo_{qr}\right)_*\left(\left(\homo_{rq}\right)_{*}(h)\right)$
 and using the fact that
 \begin{equation*}
 \cri(h)=\cri\left(\left(\homo_{qr}\right)_*\circ 
 \left(\homo_{rq}\right)_{*}(h)\right)
\end{equation*}
yields also the left inequality.
Here we use that 
$\homo_{qr}\circ \homo_{rq}: \Omega_pM
\longrightarrow \Omega_pM$ is 
homotopic to the identity map.
\end{proof}
\begin{lemma}
\label{lem:cri-two-pi}
Let $(M,g)$ be a compact Riemannian manifold of positive sectional
curvature $K\ge 1,$ 
and $p \in M:$
Let $h_p \in H_k\left(\Omega_p M;\field\right)$ 
be a non-zero homology class.

\smallskip

(a) If $k=n-1$ then $\sqrt{2\cri (h_p)} \le 2\pi.$

\smallskip

(b) If $k \in \{1,2,\ldots, n-2\}$ then $\sqrt{2\cri (h_p)} \le \pi.$

\end{lemma}
\begin{proof}
 Choose a sequence $p_j$ of regular points of $\exp_p$ with
 $\lim_{j\to \infty}d(p,p_j)=0,$ 
 cf. for example \cite[Cor. 18.2]{Milnor1963}.
Then $E: \Omega_{pp_j}M \longrightarrow \R$
is a Morse function for any $j \ge 1.$
Morse theory implies that there are 
 finitely many geodesics $c_{j,1},\ldots,c_{j,r(j)}\in \Omega_{pp_j}(M)$
 such that there exists one 
geodesic $c_{j,1}$ with \\
$\sqrt{2 \cri\left(\left(\homo_{p_jp}\right)_*(h_p)\right)}=
L\left(c_{j,1}\right)=
 \max\left\{L\left(c_{j,1}\right), \ldots,L\left(c_{j,r(j)}
\right)\right\}$
 and
 $\indo(c_j)=k\,,$ 
 cf. Section~\ref{sec:morse}.
 
 (a) Since $K \ge 1$ and  $k=n-1$ we conclude from
 Rauch comparison: $L(c_{j,1})\le 2\pi,$
 cf. \cite[ch. 2.6]{Klingenberg1995}.
 Hence $\cri(h)=\lim_{j\to \infty}\cri\left(\left(f_{p_jp}\right)_*(h_p)\right)
  \le 2\pi^2.$
  
  (b) Since $K \ge 1$ and $1 \le k \le n-2$ we conclude from
 Rauch comparison: $L(c_{j,1})\le \pi,$
 cf. \cite[ch. 2.6]{Klingenberg1995}.
 Hence $\cri(h)=\lim_{j\to \infty}\cri\left(\left(f_{p_jp}\right)_*(h_p)\right)
  \le \pi^2/2.$
 \end{proof}
For a geodesic $c\in \Omega_{pq}M$ with $L(c)>\pi$ on a
Riemannian manifold with $K \ge 1$
we construct a mapping 
$f_c: D^{n-1}\longrightarrow  \Omega_{pq}^{\le a}M$
which defines a relative homology class
$[f_c]\in H_{n-1}\left(\Omega^{< E(c)}M\cup{c}, 
\Omega^{<E(c)}M;\z_2\right).$
If $E: \Omega_{pq}M\longrightarrow \R$
is a Morse function, this relative homology
class equals the homology class
$[F_c]$ defined using the Morse Lemma,
cf. Section~\ref{sec:morse} and
Lemma~\ref{lem:fcv-properties}.

This construction is motivated by a construction
for symmetric spaces called 
{\em cutting accross the corners}
for example in~\cite[Sec.10]{Bott-Samelson1958}.
In case of the sphere $S^n$ and an
arc $c:[0,1]\longrightarrow S^n$ of a great circle
of length $\pi\le L(c)\le 2\pi$ one
starts with a cycle represented by
the curves of length $L(c)$ formed by
half-great circles through
$p=c(0)$ and $-p=c(\pi/L(c))$ 
and the arc $c(t),t\in[\pi/L(c),1].$
Hence all curves except $c$ are non-smooth
geodesic polygons, which can be shortened
by cutting accross the corners
or using
the negative gradient flow of $E.$
\begin{definition}
 \label{def:negative-cycle}
 Let $(M,g)$ be a compact and simply-connected Riemannian manifold
 of positive sectional curvature $K \ge 1$
 and {\em focal radius} $\rho.$
 Let $c:[0,1]\longrightarrow M$ be a geodesic joining
 $p=c(0)$ and $q=c(1)$ with length $L(c)\ge \pi.$ 
 We define a mapping
 \begin{equation}
  \label{eq:fcvdn}
  f_c: D^{n-1}(\rho) \longrightarrow \Omega_{pq}M
 \end{equation}
as follows:
We identify the $(n-1)$-dimensional
disc $D^{n-1}(\rho)=\{x\in \R^{n-1};$
$ \|x\|^2\le \rho\}$ with the set 
$D^{n-1}(\rho)=\left\{w \in T_pM\,;\, 
\langle w,c'(0)\rangle=0,\|w\|\le \rho \right\}$
of tangent vectors at $p$ orthogonal to $c'(0)$
with norm $\le \rho.$
Let
$\widetilde{f}_c: D^{n-1}(\rho)
\longrightarrow \Omega_{pq}M$ be defined by
 \begin{equation}
  \label{eq:fc-one}
  \widetilde{f}_c(v)(t)=
  \left\{
  \begin{array}{ccc}
   \exp_{c(t)}\left(g_v(t)V_v(t)\right)&;& 0 \le t \le \pi/L(c)\\
   c(t) &;&\pi/L(c)\le t \le 1
  \end{array}
  \right.\,.
 \end{equation}
 Here $V_v=V_v(t)$ is the parallel unit vector field along 
 $c$ with $V_v(0)=v/\|v\|, v \in
 D^{n-1}(\rho).$
The function $g_v:[0,1]\longrightarrow (-\pi/2,\pi/2), 
v \in D^{n-1}(\rho)$ is the uniquely determined
smooth function with $
g_v(0)=0$ and $g_v\left(\pi/(2L(c)) \right)=\|v\|$
such that the following holds:
For the standard round metric on $S^n$ let 
$\widetilde{\phi}_c: D^{n-1}(\pi/2)\longrightarrow
\Omega_{pq}S^n$
be the map $\widetilde{f}_c$ introduced above
in Equation~\eqref{eq:fc-one}.
Then for a parametrization $c:[0,1]\longrightarrow S^n$ 
of a part of a great circle on the round sphere $S^n$ 
proportional to arc length 
with $L(c)\ge\pi$ 
the 
curves $t \in \left[0,\pi/L(c)\right] \longmapsto
\widetilde{\phi}_c(v)(t)\in S^n$ 
are half great circles joining $p$ and 
$-p=c\left(\pi/L(c)\right)$
with
$\widetilde{\phi}_c(v)\left(\pi/(2L(c))\right)=
\exp_{c\left(\pi/(2L(c))\right)}
\left(\|v\|V_v(\pi/(2L(c)))\right)$
up to parametrization.

Let $t \in [0,1]\longrightarrow 
\overline{f}_c(v)(t)\in M$ be the reparametrization
of $t \in [0,1]
\longrightarrow \widetilde{f}_c(v)(t)\in M$
such that the following holds:
$\overline{f}_c(v)(t)=c(t)$ for all $\pi/L(c)\le t\le 1$
and $
t \in [0,L(c)/\pi] \longmapsto
\overline{f}_c(v)(t) \in M$ 
is parametrized proportional to arc length.
Denote by $\Phi^s: \Omega_{pq}M \longrightarrow \Omega_{pq}M, s>0$ the
{\em negative gradient flow} of $E.$ I.e. 
\begin{equation}
 \label{eq:negative-gradient}
 \left.\frac{d \Phi^s(\sigma)}{ds}\right|_{s=t}=-\grad E\left(\Phi^s(\sigma)\right)\,.
\end{equation}
Then finally $f_c(v):=\Phi^1 \left(\overline{f}_c(v)\right)$
and $\phi_c(v):=\Phi^1 \left(\overline{\phi}_c(v)\right)$
\end{definition}
\begin{remark}
\rm
 \label{rem:spherical-trigonometry}
 The formulae from spherical trigonometry imply:
 $$g_{v}(t)=\arctan \left(\sin\left(L(c) t\right) \tan \|v\|\right)\,.$$
\end{remark}
It is an observation by Itokawa \& Kobayashi~\cite[p.10--11]{Itokawa2008}
that using a consequence of a Corollary from the so-called
Rauch's second comparison result, 
cf. Section~\ref{sec:morse}
one can compare
the lengths $L\left(\overline{f}_c(v)\right)$
and $L\left(\overline{\phi}_c(v)\right).$
Compare also Corollary~\ref{cor:rauch} which also deals
with the rigidity case.
\begin{lemma}
 \label{lem:fcv-properties}
 Let $(M,g)$ be a compact and simply-connected Riemannian manifold
 of positive sectional curvature $K \ge 1$
 and {\em focal radius} $\rho.$
 Let $c:[0,1]\longrightarrow M$ be a geodesic joining
 $p=c(0)$ and $q=c(1)$ with length $L(c)\ge\pi,$
 i.e. $E(c)\ge \pi^2/2.$
Let $f_c: D^{n-1}(\rho)\longrightarrow \Omega_{pq}M$
be the mapping introduced
 in Definition~\ref{def:negative-cycle}.

 (a) The mapping $f_c$  
 satisfies the following property: For all 
 $v \in D^ {n-1}(\rho):$
 \begin{equation}
  \label{eq:fcv-property}
  E\left(f_c(v)\right) \le 
  E\left(f_c(0)\right)=E(c)\,.
 \end{equation}
 
 If $L(c)>\pi$ then there exists $\epsilon >0,$
such that 
 \begin{equation}
   \label{eq:fcv-property1}
   E\left(f_c(v)\right) <
   E\left(f_c(0)\right)=E(c)
  \end{equation}
  for all $v \in D^{n-1}(\rho), v\not=0$
	and 
	\begin{equation}
   \label{eq:fcv-property1a}
   E\left(f_c(v)\right) \le
   E(c)-\epsilon
  \end{equation}
	for all $v \in \partial D^{n-1}(\rho)=
	\{v \in D^{n-1}(\rho)\,;\, \|v\|=\rho.\}.$
	Hence
 $f_c$ defines a relative homology class:
 \begin{equation}
 \label{eq:fc-homology-class}
 \left[f_c\right]\in
 H_{n-1}\left(\Omega_{pq}M,
 \Omega^{\le E(c)-\epsilon}_{pq}M;\z_2\right)\,.
 \end{equation}
 
 (b) If $L(c)\in (\pi,2\pi)$ and
 if $q$ is a regular point for $\exp_p$
 and if $\indo(c) =n-1$ then
there is an $\epsilon>0$ such that for
all $\delta\in(0,\epsilon]:$
\begin{equation}
 \label{eq:fc-homology-class-nonzero}
 0\not=\left[f_c\right]\in
 H_{n-1}\left(\Omega^{\le E(c)+\delta}_{pq}M,
 \Omega^{\le E(c)-\delta}_{pq}M;\z_2\right)\,.
 \end{equation}
 
 (c)  Let $q$ be a regular point for $\exp_p$
 and let $0\not= h \in H_{n-1}\left(\Omega_{pq}M;\z_2\right)$
 with critical value
 $a:=\cri (h)\in (\pi^2/2,2\pi^2].$ Then there is 
 a geodesic $c \in \Omega_{pq}M$ with
 $E(c)=a, \indo(c)=n-1$ 
 and an $\epsilon>0$ such that
 the mapping $f_c$ defines a non-trivial relative
 homology class
 \begin{equation}
  \label{eq:fc-homology-class-nonzero}
  0\not=\left[f_c\right]\in
  H_{n-1}\left(\Omega_{pq}M,\Omega^{\le a-\delta}_{pq}M;\z_2\right)
  \end{equation}
	for all $\delta \in (0,\epsilon].$
 \end{lemma}
\begin{proof}
(a) We denote by $f_c,\widetilde{f}_c,\overline{f}_c: D^{n-1}(\rho)\longrightarrow \Omega_{pq}M$
the mappings introduced in Definition~\ref{def:negative-cycle} for the
Riemannian manifold $(M,g)$ and
by $\phi_c,\widetilde{\phi}_c,\overline{\phi}_c: D^{n-1}(\pi/2)\longrightarrow \Omega_{pq}S^n$
the mappings associated to the standard round metric of constant sectional
curvature $1$ on the $n$-sphere $S^n.$
Hence
$t \in \left[0,\pi/L(c)\right]\longmapsto \widetilde{\phi}_c(v)(t)\in S^n$
is a parametrization of a half great circle and
$t \in \left[0,\pi/L(c)\right]\longmapsto \overline{\phi}_c(v)(t)\in S^n$
is its parametrization proportional to arc length.
Hence $L\left(\widetilde{\phi}_c(v)\right)=
L\left(\overline{\phi}_c(v)\right)=L(c).$
Then 
Corollary~\ref{cor:rauch}
implies:
\begin{eqnarray*}
L\left(\widetilde{f}_c(v)
\left|\left[0,\pi/L(c)\right]\right.\right)=
 L\left(\overline{f}_c(v)
\left|\left[0,\pi/L(c)\right]\right.\right)
 \\
 \le
L\left(\widetilde{\phi}_c(v)
\left|\left[0,\pi/L(c)\right]\right.\right)=
 L\left(\overline{\phi}_c(v)
\left|\left[0,\pi/L(c)\right]\right.\right)=\pi
\end{eqnarray*}
since $K\ge 1.$ 
This shows Equation~\eqref{eq:fcv-property}.
Now assume $L(c)\in (\pi,2\pi].$ Then
for $v \not=0$ the curves 
$\overline{f}_c(v)$ are not smooth at $t_1=\pi/L(c),$
hence 
\begin{equation*}
 E\left(\Phi^1
\left(\overline{f}_c(v)\right)\right)< 
 E\left(\overline{f}_c(v)\right)\le E(c)\,.
\end{equation*}
This proves  Equation~\eqref{eq:fcv-property1}
and Equation~\eqref{eq:fc-homology-class}. 

\smallskip

(b) We define homotopies
$f_{c,u}, \widetilde{f}_{c,u}, \overline{f}_{c,u}: D^{n-1}(\rho) \longrightarrow \Omega_{pq}M$
with $u\ge \pi/L(c)$ 
of the mappings 
$f_c=f_{c,\pi/L(c)}, \widetilde{f}_c
=\widetilde{f}_{c,\pi/L(c)}, 
\overline{f}_c=\overline{f}_{c,\pi/L(c)}$
resp. homotopies
$\phi_{c,u}, \widetilde{\phi}_{c,u}, \overline{\phi}_{c,u}: D^{n-1}(\rho) \longrightarrow \Omega_{pq}M$
with $u\ge \pi/L(c)$ 
of the mappings
$\phi_c=\phi_{c,\pi/L(c)}, \widetilde{\phi}_c
=\widetilde{\phi}_{c,\pi/L(c)}, 
\overline{\phi}_c=\overline{f}_{c,\pi/L(c)}$
as follows:
For $v \in D^{n-1}(\rho), u \in (\pi/L(c),2\pi/L(c))$ let
\begin{equation}
 \label{eq:gvu}
 g_{v,u}(t)=
  \arctan\left(\sin\left(L(c)t\right)
\frac{\tan \|v\|}{\sin \left(L(c)u/2\right)}\right)\,.
\end{equation}
Let
\begin{equation*}
\widetilde{f}_{c,u}(t)=
\left\{
\begin{array}{ccc}
 \exp_{c(t)}\left(g_{v,u}(t)V_v(t)\right) &;& 0\le t \le u/2\\
  \exp_{c(t)}\left(g_{v,u}(u-t)V_v(t)\right) &;& u/2\le t \le u\\
 c(t)&;& u\le t \le 1
\end{array}
\right. \,.
\end{equation*}
Let $t \in [0,u]\longrightarrow 
\overline{f}_{c,u}(v)(t)$ be the 
parametrization of $t \in [0,u]
\longrightarrow \widetilde{f}_{c,u}(v)(t)$ proportional
to arc length. 
We denote by $\overline{\phi}_{c,u}(v), \widetilde{\phi}_{c,u}(v)$
the corresponding mappings for the standard round metric on $S^n.$
Using the formulae from spherical trigonometry one can check
that the function $g_{v,u}(t)$ is the uniquely defined smooth
function such that
$t \in [0,u]\longmapsto 
\widetilde{\phi}_{c,u}(v)(t) =$ $
=\exp_{c(t)}\left(g_{v,u}(t)E_v(t)\right)\in S^n$
is a parametrization of the arc of a great circle
joining $p=c(0)$ and
$\exp_{c(u/2)}\left(\|v\|V_v(u/2)\right).$
From spherical trigonometry we conclude for the length $L(u,v)=L\left(\overline{\phi}_{c,u}(v)\left|
 \left[0,u/2\right]\right.\right)=
  L\left(\overline{\phi}_{c,u}(v)\left|
   \left[u/2,u\right]\right.\right)
 $
of the
great circle arcs:
\begin{equation*}
L(u,v) =
 \arccos \left(\cos\left(L(c)\frac{u}{2}\right)
 \cos\|v\|\right)\,.
\end{equation*}
We compute:
\begin{equation*}
 \left.\frac{\partial^2 L 
\left(\overline{\phi}_{c,u}(v)\right)}{\partial v_i
 ^2}\right|_{v_i=0}=
 \frac{\cos\left(L(c) u/2\right)}{\sqrt{1-\cos^2\left(L(c)u/2\right)}}\,.
\end{equation*}
We conclude that for $u > \pi/L(c)$ the Hessian
$d^2\left(E\circ \overline{\phi}_{c,u}\right)(0)$ at the point 
$0 \in D^{n-1}\left(\pi/2\right)$
is negative definite, i.e. 
$c=\overline{\phi}_{c,u}(0)$ is a non-degenerate maximum point of
the restriction
$E: \overline{\phi}_{c,u}\left(D^{n-1}(\pi/2)\right) \longrightarrow \R.$
Using again Corollary~\ref{cor:rauch} 
we obtain, that
$E\left(\overline{f}_{c,u}(v)\right) 
\le E \left(\overline{\phi}_{c,u}(v)\right)$
for all $u \in (\pi/L(c),2\pi/L(c)),$
$v \in D^ {n-1}(\rho).$
Therefore $c=\overline{f}_{c,u}(0)$ is also
a non-degenerate maximum point of the 
restriction of
the energy functional 
$E: \overline{f}_{c,u}\left(D^{n-1}(\rho)\right) \longrightarrow \R$ 
to the local $(n-1)$-dimensional submanifold 
$\overline{f}_{c,u}\left(D^{n-1}(\rho)\right).$
Consider
$\Phi^1\circ \overline{f}_{c,u}:D^{n-1}(\rho)\longrightarrow \Omega_{pq}^{\le E(c)}M, u \ge \pi/L(c).$
Then there is $\epsilon >0$ such that
$\Phi^1\left(\overline{f}_{c,u}\left(\partial D^{n-1}(\rho)\right)\right)
\subset\Omega_{pq}^{\le E(c)-\epsilon}M$
for all $u \ge \pi/L(c).$

This shows that 
$f_{c,u}$ defines a non-zero
homology class 
\begin{equation}
\label{eq:fcu-nonzero}
0\not=\left[f_{c,u}\right]\in
H_{n-1}\left(\Omega_{pq}^{\le E(c)+\delta}M,
\Omega_{pq}^ {\le E(c)-\delta}M;\z_2\right)
\end{equation}
for all $\delta \in (0,\epsilon]$
and $u >\pi/L(c).$ This holds by the Morse-Lemma
since $c$ is a non-degenerate critical point of 
$E: \Omega_{pq}^ {\le E(c)+\epsilon}M \longrightarrow \R$ with $\indo c=n-1,$
cf. Section~\ref{sec:morse}.
Define
\begin{equation*}
f_{c,u}:=\Phi^1 \left(\overline{f}_{c,u}\right):
\left(D^{n-1}(\rho),\partial D^{n-1}(\rho)\right)
\longrightarrow 
\left(\Omega_{pq}^{\le E(c)+\delta}M, \Omega_{pq}^{\le E(c)-\delta}M\right)\,.
\end{equation*}
for $u \ge \pi/L(c).$ 
The clue here is that by introducing
the flow $\Phi^1$ this homotopy is also defined
for $u=\pi/L(c).$
Then Equation~\eqref{eq:fcu-nonzero} implies:
\begin{equation}
\label{eq:fcu-nonzero1}
0\not=\left[f_{c}\right]=
\left[f_{c,u}\right]\in
H_{n-1}\left(\Omega_{pq}^{\le E(c)+\delta}M,\Omega_{pq}^ {\le E(c)-\delta}M;\z_2\right)
\end{equation}
for some $u>\pi/L(c).$

\smallskip

(c) 
We conclude from Part (b) and 
Section~\ref{sec:morse}:
There are finitely many geodesics $c_1,\ldots, c_{r}\in \Omega_{pq}M$
with $\indo(c_1)=\ldots=\indo(c_{r})=n-1$ and
$E(c_1)= E(c_2)=\ldots = E(c_{r})=a$
such that the class $h$ as relative class
in $H_{n-1}\left(\Omega_{pq}^{\le a+\delta}M,
\Omega_{pq}^{\le a-\delta}M;\z_2\right)$
for some $\delta>0$
has the following representation:
For sufficiently small
$\epsilon >0$ the 
mappings 
$f_{c_j}: \left(D^ {n-1}(\rho), \partial D^ {n-1}(\rho)\right)\longrightarrow
\left(\Omega_{pq}^ {\le a +\delta}M,\Omega_{pq}^ {\le a -\delta}M\right), j=1,2,\ldots,r$
represent non-trivial relative
homology classes
$0 \not=\left[f_{c_j}\right]
=\left[F_{c_j}\right]
\in H_{n-1}
\left(\Omega_{pq}^ {\le a +\epsilon}M,\right.$ $\left.\Omega_{pq}^ {\le a -\epsilon}M;\z_2\right),
j=1,2,\ldots,r$ 
for all $\delta \in (0,\epsilon].$
Hence 
in $H_{n-1}\left(\Omega_{pq}^{\le a+\delta}M,
\Omega_{pq}^{\le a-\delta}M;\z_2\right)$
we have the following representation:
$h =[f_{c_1}]+\ldots+[f_{c_r}].$
Hence the image of $h$ under the homomorphism
$$H_{n-1}
\left(\Omega_{pq}^{\le a+\delta}M;\z_2\right)\longrightarrow 
H_{n-1}
\left(\Omega_{pq}M, 
\Omega_{pq}^{\le a-\delta}M;\z_2 \right)$$
is mapped onto 
$\alpha_*\left(\sum_{j=1}^r
 \left[f_{c_j}\right]\right).$
Here 
$$\alpha_*: H_{n-1}\left(\Omega_{pq}^{\le a+\delta}M,
\Omega_{pq}^ {\le a-\delta}M;
\z_2\right)
\longrightarrow 
H_{n-1}\left(\Omega_{pq}M, \Omega_{pq}^{\le a-\delta}M; \z_2\right)$$
is the homomorphism induced by the inclusion. If
$\alpha_*\left[f_{c_j}\right]=0$ for all $j=1,2,\ldots,r$
then $h$ lies in the image
of the homomorphism
$$ H_{n-1}\left(\Omega^ {\le a-\delta}_{pq}M;
\z_2\right)
\longrightarrow 
H_{n-1}\left(\Omega_{pq}M;\z_2\right),$$
i.e. $\cri(h)\le a-\delta.$ This is a contradiction.
\end{proof}
\begin{lemma}
\label{lem:path}
Let $(M,g)$ be a compact and simply-connected
Riemannian manifold 
of positive sectional curvature $K\ge 1:$

(a)
Let $\gamma\in \Omega_{qp}M$ be a geodesic
of length $L(\gamma)=\pi$ 
joining $q=\gamma(0)$ and $p=\gamma(1).$ 
If $\gamma $ is not minimal, i.e. if
$d(p,q)<\pi$ then there is a continuous path
$s \in [0,1]\longmapsto \gamma_s \in \Omega_{qp}M$
with $\gamma_0=\gamma, E\left(\gamma_s\right)\le\pi^2/2$
for $s \in (0,1)$ and $E(\gamma_1)<\pi^2/2.$
If $E(\gamma_{s_1})=\pi^2/2,$ then
$E(\gamma_s)=\pi^2/2$ for all $s \in [0,s_1].$

(b) If $c\in \Omega_pM$ is a geodesic loop with
$p=c(0),q=c(1/2)$ and length $L(c)=2\pi$ then
there is a path $s \in [0,1]
\longmapsto c_s\in\Omega_p^{\le 2\pi^2}M$ 
such that the following
holds: For $s \in [0,1/2]$ the curves
$c_s$ are geodesics with
$L(c_s)=2\pi, c_s(1/2)=q.$
And $E(c_s)< 2 \pi^2$ for all $s \in (1/2,1].$
\end{lemma}
\begin{proof}
(a)
Define the subset
\begin{equation*}
A:=\left\{y \in T_qM; \|y\|=\pi,
 \exp_q(y)=p\right\}\,.
\end{equation*}
of the
set $T^{\pi}_q M:=\{y \in T_qM; \|y\|=\pi,\}$
of tangent vectors at $q$ of norm $\pi.$
For $y \in A$ denote by
\begin{equation*}
c_y: [0,1]\longrightarrow M, c_y(t)=\exp_q(ty)
\end{equation*}
the geodesic of length $\pi$ joining
$q=c_y(0)$ and $p=c_y(1)=\exp_q(y).$
For $y \in A$ one can 
consider the mapping $f_{c_y}: D^{n-1}(\rho)
\longrightarrow\Omega_{qp}^{\le \pi^2/2}M$
defined in Definition~\ref{def:negative-cycle},
compare Lemma~\ref{lem:fcv-properties} resp.
Equation~\eqref{eq:fcv-property}.
For $y \in A$ define
\begin{equation*}
\eta_y:=\sup \left\{\eta \in [0,\rho]\,;\,
E\left(f_{c_y}
\left(D^{n-1}(\eta)\right)\right)=
\kappa^2/2\right\}\,.
\end{equation*}
If $\eta_y>0$ for $y \in A$ then the
curves
$t \in [0,1]\longmapsto f_{c_y}(v)(t)\in M$ 
with $\|v\|\le \eta_y$ have energy $\pi^2/2$
and are therefore by construction
of $f_{c_y}$ geodesics. This implies that there
is a neigborhood $U$ of $y$ in $A$ such that
for $z \in U:$
$\eta_z>0.$ Hence $y$ is a point in the interior
$\mathring{A}$ of the set $A.$
By assumption $c'(0) \in A,$ let
$B \subset A$ be the path-component of
$A$ containing $c'(0).$
$B$ is a closed subset of
$T_q^{\pi}M=\{v\in T_qM;\|v\|=\pi\}.$
Now assume that for all $y \in B: \eta_y>0.$
Then the set $B \subset T_q^{\pi}M$ is also
open, i.e. $B=T_q^{\pi}M.$ 
Hence $d(q,p)=\pi$ in contradiction
to the assumption. Therefore there is 
$y \in B$ with $\eta_y=0.$ Then one
can construct a path 
$s \in [0,1]\longrightarrow 
\Omega_{qp}M^{\le \pi^2/2}M$ 
with the following
properties: The restriction
$s \in [0,1/2]
\mapsto \alpha_s\in \Omega_{qp}M$ is
a path consisting of geodesics of length
$\pi$ joining $\alpha_0=c$ and
$\alpha_{1/2}=c_{y},$
i.e. $\alpha'_s(0)\in B.$
And 
$s \in [1/2,1]\mapsto \alpha_s(v)=f_{c_y}(v(s))
\in \Omega_{qp}^{\le \pi^2/2}M$ is a
path satisfying 
$E\left(\alpha_s\right)<\pi^2/2$ 
for all $s\in (1/2,1]$ and a path
$s\in [1/2,1]\mapsto v(s)\in
D^{n-1}(\rho).$
This is possible since $\eta_y=0$
and by the equality discussion in
Corollary~\ref{cor:rauch}.

\smallskip

(b) Let $\delta: T_q^\pi M \times T_q^\pi M
\longrightarrow \R$ be the distance
induced by the Riemannian metric $g_q/\pi^2$
on the sphere $T^{\pi}_qM.$
This is the standard round metric of
sectional curvature $1$ on $T_q^{\pi}M.$
We define the numbers
\begin{equation}
\label{eq:eta}
\theta_{\pm}:=\sup \left\{
\theta \ge 0; \forall w\in T_q^\pi M, 
\delta(\pm c'(1/2), w)\le \theta: \exp_q(w)=p
\right\}\,.
\end{equation}
Then it follows that there are
$y_{\pm}\in T_q^{\pi}M$ with
$\delta(\pm c'(1/2),y_{\pm})=\theta_{\pm}$
and $\eta_{y_{\pm}}=0.$
Then $0 \le \theta_\pm < \pi$ since $d(p,q)<\pi.$
Without restriction we can assume that
$\theta_- \ge \theta_+.$ Then there exists
a geodesic $\widetilde{c} \in \Omega_pM$ with
$\delta(c'(1/2),\widetilde{c}'(1/2)=
\theta_+$ and $\eta(\widetilde{c}'(1/2)=0.$
Then $s \in [0,1/2]\mapsto c_s \in 
\Omega_p^{2\pi^2}M$ 
is a path of geodesics of
length $2\pi$
joining $c_0=c.$ In addition $c_{1/2}=\widetilde{c}$
and $c_s \in \Omega_pM$ for $s \in [1/2,1]$
is a path with $c_s(t)=\widetilde{c}(t)$
for $t \in [0,1/2]$ and $
\gamma_s=c_s \left|[1/2,1]\right.$
is a curve with $E(\gamma_s)<\pi^2/2$
for all $s \in (1/2,1]$
as constructed in part (a).
\end{proof}
\begin{lemma}
\label{lem:crifc}
Let $\left(M,g\right)$ be a compact 
and simply-connected Riemannian manifold
with positive sectional curvature $K\ge 1$ and
let $c\in \Omega_pM$
be a geodesic loop based at $p$
of length $2\pi,$ resp. energy
$2\pi^2.$ 
For sufficiently small positive $\epsilon>0$
the mapping 
$f_c: D^{n-1}(\rho)\longrightarrow \Omega_pM$
introduced in Definition~\ref{def:negative-cycle}
with $E\left(f_c(v)\right)<2\pi^2$ for all $v \not=0$
defines a relative homology class
\begin{equation}
\label{eq:fc2pi}
\left[f_c\right]\in H_{n-1}\left(\Omega_pM,
\Omega_p^{\le 2\pi-\epsilon}M;\z_2\right).
\end{equation}
If its critical value satisfies
$\cri \left(\left[f_c\right]\right)=2\pi$ then
the Riemannian manifold is a round sphere
of constant sectional curvature $1.$
\end{lemma}
\begin{proof}
If the sectional curvature is not constant, the diameter
$\diam (M,g)$ satisfies
$\diam M < \pi.$ This is a consequence of
Toponogov's maximal diameter 
theorem~\cite[Thm. 6.5]{Cheeger1975} 
resp.~\cite{Toponogov1958}.

Let $q=c(1/2), p=c(0)=c(1).$
Then $d(q,p)\le \diam (M,g)<\pi.$
We conclude from Lemma~\ref{lem:path} that there is 
a path $s \in [0,1] \longmapsto c_s \in \Omega_{p}M$
with $c_0=c, c_{1/2}=\widetilde{c}\in 
\Omega_pM$ a geodesic with
$L(\widetilde{c})=2\pi$
resp. $E(\widetilde{c})=2\pi^2$ and
$E(c_s)<2\pi^2$ for $s\in (1/2,1].$
It follows that for sufficiently small $\epsilon >0$
\begin{equation}
\label{eq:ctilde}
\left[f_c\right]=
\left[f_{\widetilde{c}}\right]
\in 
H_{n-1}\left(\Omega_p M,
\Omega_{p}^{\le 2\pi^2-\epsilon}M;\z_2\right)
\end{equation}
and
$c_s(t)=\widetilde{c}(t)$
for all $s \ge 1/2$ and all
$t \in [0,1/2].$
Then we define a homotopy 
$\left(f_{\widetilde{c}}\right)_s$ of the mapping
$f_{\widetilde{c}}:D^{n-1}(\rho)\longrightarrow
\Omega_p^{\le 2\pi^2/2}M.$
At first we define a homotopy 
$\left(\overline{f}_{\widetilde{c}}\right)_s
 :D^{n-1}(\rho)\longrightarrow
 \Omega_p^{\le 2\pi^2}M, s \in [0,1]$
of the mapping
$\overline{f}_{\widetilde{c}}: D^{n-1}
\longrightarrow \Omega_p^{\le 2\pi^2}M:$ 
\begin{equation*}
\overline{f}_{\widetilde{c},s}(v)(t)=
\left\{
\begin{array}{ccc}
\overline{f}_{\widetilde{c}}(v)(t)&;& 0\le t \le 1/2\\
c_s(1/2+2t)&;& 1/2 \le t \le 1
\end{array}
\right.\,.
\end{equation*}
Then $E\left(\overline{f}_{\widetilde{c},s}(v)
\right)\le 2\pi^2$
for all $s\in [0,1]$ and
$E\left(\overline{f}_{\widetilde{c},s}(v)
\right)< 2\pi^2$
for all $v \in D^{n-1}(\rho), s> 1/2.$
We use the negative gradient flow 
$\Phi^s: \Omega_pM \longrightarrow \Omega_pM, s\ge 0$
of the energy functional to define:
$f_{c,s}(v):=\Phi^1 \left(\overline{f}_{c,s}(v)\right).$
Then $E\left(f_{c,s}(v)\right)\le 2\pi^2$ 
for all $v \in D^{n-1}$ and $s \in [0,1]$
and $E\left(f_{c,s}(v)\right)< 2\pi^2$
for all $v \in D^{n-1}, s> 1/2.$
Hence 
the homotopy 
$f_{c,s},s \in [0,1]$
shows
$\cri \left(\left[f_c\right]\right)=
\cri \left(\left[f_{\widetilde{c}}\right]\right)
<2\pi.$
\end{proof}
\section{Proof of Theorem 1}
\label{sec:proof-theorem-one}
 (a) Lemma~\ref{lem:cri-two-pi} shows that $\crimp \le 2\pi$
for all $p \in M.$ Hence also $\crim \le 2\pi$
by Equation~\eqref{eq:crlp}.
 
 \smallskip
 
 (b) Let $h_p \in H_k\left(\Omega_pM,p;\z_2\right)$
 with
 $\crimp = \cri (h)>\pi.$ 
Lemma~\ref{lem:cri-two-pi}(b) implies 
 $k=n-1.$ Hence the manifold $M$ is $(n-1)$-connected
 and hence by the solution of the
 Poincar\'e conjecture homeomorphic to $S^n\,.$
If $\crim >\pi$ then $\crimp >\pi$ by 
Equation~\eqref{eq:crlp}.
 
 \smallskip
 
 (c)
If $\crim =2\pi$ Equation~\eqref{eq:crlp} and part (a)
imply that $\crimp=2\pi$ for all $p \in M.$
Hence we assume now $\crimp=2\pi$ for some $p \in M.$
 Part (b) implies that 
 $M$ is homeomorphic to $S^n$ and 
 $\crimp=\cri\left(h_p\right),$ where 
 $h_p\in H_{n-1}\left(\Omega_{pq}M;\z_2\right)\cong \z_2$
 is non-zero.
 Choose a sequence $p_j$ of regular points of $\exp_p$ with
 $\lim_{j\to \infty}d(p,p_j)=0,$ 
 cf. for example~\cite[Cor. 18.2]{Milnor1963}.
Then $E: \Omega_{pp_j}M \longrightarrow \R$ 
is a Morse function for any $j \ge 1,$
cf. Section~\ref{sec:morse}.
This holds since $p_j$
as a  regular point of $\exp_p$ is not conjugate to $p$ along any geodesic
joining $p$ and $q.$ 
Using the homotopy equivalence 
$\homo_{p_jp}: \Omega_{p}M \longrightarrow \Omega_{pp_j}M$
introduced in Lemma~\ref{lem:homo} we define
\begin{equation}
 \label{eq:aj}
 a_j:=\cri\left(\left(\homo_{p_jp}\right)_*\left(h_p\right)\right)\,,
\end{equation}
here
$\left(\homo_{p_jp}\right)_*\left(
h_p\right)
\in H_{n-1}\left(\Omega_{pp_j}M;\z_2\right).$

Then Equation~\eqref{eq:crisigma} implies that 
$\lim_{j\to \infty}a_j=2\pi^2$ since 
$p=\lim_{j\to \infty}p_j.$

Morse theory implies that 
for any $j \ge 1$
there is a geodesic 
$c_j\in \Omega_{pp_j}M$
and $\epsilon_j>0$ with $a_j=E(c_j), 
\indo(c_j)=n-1$ 
such that the mapping
\begin{equation}
 \label{eq:fjdn}
 f_j=f_{c_j}: D^{n-1}(\rho)
\longrightarrow \Omega_{pp_j}M
\end{equation}
introduced in Definition~\ref{def:negative-cycle} resp.
Equation~\eqref{eq:fcvdn} satisfies:
$ a_j-\epsilon_j=$\\
$\max\left\{F\left(f_j(v)\right); 
 v \in \partial D^ {n-1}\right\},$
cf. Lemma~\ref{lem:fcv-properties}.
In addition we can assume that for any 
$\delta \in (0,\epsilon_j]$ the relative
homology class
\begin{equation}
 \label{eq:hjnonzero}
0\not=\left[f_j\right] \in H_{n-1}\left(\Omega_{pp_j}M, 
\Omega_{pp_j}^{\le a_j-\delta}M;\z_2\right)
\end{equation}
is non-zero. This follows from Lemma~\ref{lem:fcv-properties} (b) and (c).

A subsequence $\left(c_j\right)_j$ converges to 
a geodesic loop $c \in \Omega_pM,$ 
which we also denote by 
$\left(c_j\right)_j.$
Then we obtain for the maps
$f_j, j\in \n$ and the map
$f_c: D^{n-1}(\rho)\longrightarrow \Omega_p M$
the following statements:
\begin {equation}
 \label{eq:fj-d}
 \lim_{j\to \infty}\sup \left\{
 d\left(f_j(v)(t), f_c(v)(t)\right), t\in [0,1]\right\}=0
\end {equation}
and
\begin {equation}
\label{eq:Lfcj}
 E\left(f_c(v)\right)=\lim_{j\to \infty} E\left(f_j(v)\right)\,.
 \end {equation}
Since 
$\epsilon':=2\pi -\max\left\{F\left(f_c(v)\right)\,;\, 
 v \in \partial D^ {n-1}\right\}>0
$
 by Lemma~\ref{lem:fcv-properties} and 
 $\epsilon' =\lim_{j\to \infty} \epsilon_j$
 we obtain that 
$\epsilon_j\ge \epsilon'/2$ for all $j\ge j_1$
for some $j_1 \in \n.$
 
 Then 
\begin{equation}
 \label{eq:fppj}
 \overline{f}_{\epsilon}=[f_c]=
 \left[\homo_{pp_j}\circ f_j\right]\in H_{n-1}
 \left(\Omega_pM,\Omega_p^{\le 2\pi-\delta/4}M\right)
 \not=0
\end{equation}
for all $\delta \in (0,\epsilon'].$
Equation~\eqref{eq:fj-d} implies that 
$f_c$ and $\homo_{p_jp
}\circ f_j$ are homotopic
for sufficiently large $j,$ which implies
the second equality of Equation~\eqref{eq:fppj}. 
Since Equation~\eqref{eq:fppj} holds for any 
$\delta
\in (0,\epsilon']$ we conclude for the critical
value $\cri \left(\overline{f}_{\epsilon}\right)=2\pi.$ 
Lemma~\ref{lem:crifc} implies that
the Riemannian metric $g$ has constant sectional
curvature. This finishes the proof of
Theorem~\ref{thm:one}.
\section{Morse condition and critical length}
\label{sec:metric}

Instead of comparing the sectional
curvature $K$ of the Riemannian metric $g$
with the sectional curvature $1$ of the 
standard metric $g_1$ on $S^n$ we can
compare the metrics $g$ and $g_1$ on $S^n.$
For existence results of closed geodesics
on spheres instead of the assumption
$1 \le K \le 4$ several authors considered
the assumption $g_1/4< g< g_1,$ the 
so-called {\em Morse condition,}
cf.~\cite[Sec.3]{BTZ1983a}.
Instead of the lower bound 
$K \ge 1$ for the sectional
curvature in Theorem~\ref{thm:one}
we use the assumption $g \le g_1$ and
obtain the following
\begin{proposition}
\label{pro:metric}
Let $g$ be a Riemannian metric on the 
$n$-sphere $S^n$ and $g_1$ be the standard 
Riemannian metric of constant sectional curvature
$1$ on $S^n$ such that $g \le g_1.$

\smallskip

(a) 
For all $p \in M:$ $\crisp\le 2 \pi$
and $\cris \le 2\pi.$

\smallskip

(b) If $\cris =2\pi$ then $g=g_1.$
\end{proposition}
\begin{proof}
(a) The assumption $g\le g_1$ implies
that for any 
$h \in H_k\left(X;X^{\le b}\right)$
the critical value $\cri(h)$ with respect
to the metric $g$ satisfies
$\cri(h)\le\cri_1(h),$
where $\cri_1(h)$ is the critical value
with respect to the standard metric.
Since $\crisp=\sqrt{2 \cri(h_p)}$
for a generator $h_p \in H_{n-1}\left(\Omega_pS^n;\z_2\right)$
resp.
$\cris=\sqrt{2 \cri(h)}$ for a generator
$h \in H_{n-1}\left(\Lambda S^n,
\Lambda^0S^n;\z_2\right)$
the claim follows.

\smallskip

(b) Assume that for some $p \in S^n$
there is a tangent
vector $w \in T_{-p}S^n$ with $g(w,w)<g_1(w,w).$

Let $T_p^{2\pi}S^n:=\left\{v \in T_pS^n;
g_1(v,v)=4\pi^2\right\}.$
Then we define a map
$\Gamma: v \in T_p^{2\pi}S^n\cong S^{n-1}
\longmapsto \Gamma(v) \in \Omega_pS^n$
as follows:
Let $\gamma_v: [0,1]\longrightarrow S^n$
be the great circle parametrized proportional
to arc length with $p=\gamma_v(0)=\gamma_v(1)$
and $v=\gamma_v'(0).$ In particular 
$\gamma_v(1/2)=-\gamma_v(0)=-p.$
Let $w_1 \in T_pS^n$ be the vector
such that $w=\gamma'_{w_1}(1/2).$
We define
\begin{equation}
\label{eq:Gamma}
\Gamma(v)(t)=\left\{
\begin{array}{ccc}
\gamma_v(t) &;& 0 \le t \le 1/2\\
\gamma_{w_1}(t) &;& 1/2\le t \le 1
\end{array}\right.\,.
\end{equation}

Then the map
$(t,v) \in S^1\times S^{n-1} \longmapsto
\Gamma(v)(t)\in S^n$ defines a homotopically
non-trivial map of degree $1.$
Hence $\Gamma$ represents the 
non-trivial class $h_p.$
Choose the negative gradient flow
$\Phi^s: \Omega_p S^n \longrightarrow \Omega_p S^n$
of the energy functional $E$ with respect to
the metric $g.$
Since 
$E(\Gamma(w_1))=E(\gamma_{w_1})<2\pi^2$
by assumption and
since $\Gamma(v)$ is not smooth for any
$v\not=w_1$ we obtain
$E(\Phi^1\circ\Gamma(v))<2\pi^2$
for all $v \in T_p^{2\pi}S^n.$
Therefore
$\crisp =\sqrt{2 \,\cri(\left[\Phi^1\circ \Gamma\right])}
< 2\pi.$
\end{proof}
For a compact Riemannian manifold $(M,g)$
and a point $p \in M$ define
$d_p=\sup \{d(p.q); q \in M\}.$
Then the diameter $\diam$ is given as
$\diam = \sup \{d_p;p\in M\}$ and the
{\em radius} by $\rad = \inf \{d_p;p\in M\}.$
\begin{proposition}
\label{pro:radius}
For a compact Riemannian manifold $(M,g)$
the critical length $\crl_p(M,g)$ at the point
$p$ satisfies
$ \crl_p (M,g) \ge 2 d_p.$ 
In particular we obtain
$\inf \{\crl_p(M,g) ; p \in M\}\ge 2 \rad (M,g).$
\end{proposition}
\begin{proof}
Let $h_p \in H_{n-1}\left(\Omega_pS^n;\z_2\right)$
be the generator, hence 
$\crl_p=\sqrt{2 \cri(h_p)}.$
Choose a map 
$\xi_p : S^{n-1}\longrightarrow \Omega_pS^n$
representing $h_p.$ Then the map
$
(t,x)\in [0,1]/\{0,1\}\times S^{n-1}
\longmapsto \xi_p(x)(t) \in S^n$ is a map
of degree $1,$ in particular the map 
is surjective.
Hence $\sup\left\{L\left(\xi_p(x)\right)\,;\, x \in S^{n-1}\right\}\ge 2d_p.$
\end{proof}
\begin{example}
\label{exa:convex}
\rm
One can write the standard metric
$g_1$ on $S^2$ as warped product
$dt^2+\sin^2 t dx^2$ in the coordinates
 $(t,x)\in [0,\pi]\times [0,2\pi]/\{0,2\pi\}
 = [0,\pi]\times \z/(2\pi \z).$
 Denote by $p=(0,x)$ resp. $p'=(\pi,x)$ the 
 coordinate singularities.
Choose a one-parameter family 
of smooth functions
$f_r:[-\pi,\pi] \longrightarrow \R, 
r \in (0,1]$ 
with $f_r'(0)=1,
f_r(\pi/2)= r$ and
$f_1(t)=\sin t, 
f_r(t)=-f_r(-t), 0<f_r(t)\le \sin t,
f_r(\pi/2+t)=f_r(\pi/2-t)$
and $f''_r(t)<0
$ for all $t \in (0,\pi).$

Then the 
warped product metric
$dt^2 +f_r^2(t) dx^2$ on $(0,\pi)\times (0,2\pi)/\{0,2\pi\}$
extends smoothly to a 
Riemannian metric
$g=g_r$ on $S^2.$ 

The Gauss curvature 
$K(t,x)=-f_r''(t)/f_r(t)$ is positive
everywhere. 
The surfaces can be seen as
surfaces of revolution generated by a 
one-parameter family of convex
curves 
$t\in [0,2\pi]\longmapsto 
c_r(t)=(x(t),z(t))\in [0,r]\times [-\pi, \pi]$
parametrized by arc length, 
which
intersect at $t=0$ and $t=\pi$ the
axis of revolution orthogonally,
having distance $\le r$ from the 
axis of revolution and
such that $c_1$ is a half-circle.

For $r<1$ we have $g \le g_1$ and
$g \not= g_1.$
Since all geodesics starting from 
the point $p$ are closed and
have length $2\pi$ we obtain
$\crisp=2\pi.$ 
For any other point $q \in S^2$ the 
distance to $p$ or to $p'$ is at least 
$\pi/2.$
Therefore 
$\crl_q(S^2,g_r)
\ge \pi.$
For $t \in [0,\pi]$ the curve
$s \in [0,1] \longmapsto \gamma_t(s)=(t, 2\pi s)$
is a circle of length $2\pi f(t)$
on the surface of revolution.
Hence the map
$t \in \left([0,\pi], \{0,\pi\}\right)\longmapsto \gamma_t \in 
\left(\Lambda S^2, \Lambda^0S^2\right)$
represents the generator
$h\in H_{1}\left(\Lambda S^2, 
\Lambda^0;\z_2\right).$
Therefore $\cris \le 2\pi f(\pi/2)= 2\pi r.$
Hence we obtain a one-parameter family $g_r, r \in (0,1)$
of convex surfaces $(S^2,g_r)$ of revolution
with $\cri_q(S^2,g_r)\ge \pi$ for all
$q,$ $\cri_p(S^2,g_r)=\cri_{p'}(S^2,g_r)=2\pi,
g\le g_1$ and $\lim_{r\to 0}\cri(S^2,g_r)=0.$
This example shows that in Part (b)
of Proposition~\ref{pro:metric} it
is not sufficient to assume that 
$\crimp =2\pi.$
\end{example}

\bibliography{ClosedGeodesics}
\bibliographystyle{plain}
\end{document}